\documentclass[12 pt,twoside]{amsart}

\usepackage{amsopn}
\usepackage{amssymb}
\usepackage{amscd}

\newtheorem{theorem}{Theorem}[section]

\newtheorem{proposition}[theorem]{Proposition}
\newtheorem{corollary}[theorem]{Corollary}
\newtheorem*{Theorem}{Theorem}

\theoremstyle{definition}
\newtheorem{definition}[theorem]{Definition}

\theoremstyle{remark}
\newtheorem{remark}[theorem]{Remark}

\numberwithin{equation}{section}

\begin{document}

\title[A Birkhoff type transitivity theorem]{A Birkhoff type transitivity theorem for non-separable completely metrizable spaces with
applications to Linear Dynamics}

\author[A. Manoussos]{A. Manoussos}
\address{Fakult\"{a}t f\"{u}r Mathematik, SFB 701, Universit\"{a}t Bielefeld, Postfach 100131, D-33501 Bielefeld, Germany}
\email{amanouss@math.uni-bielefeld.de}
\urladdr{http://www.math.uni-bielefeld.de/~amanouss}
\thanks{During this research the author was fully supported by SFB 701 ``Spektrale Strukturen und
Topologische Methoden in der Mathematik" at the University of Bielefeld, Germany.}

\subjclass[2010]{Primary 47A16, 54H20; Secondary 37B99, 54H15}

\date{}

\keywords{Topological transitivity, hypercyclicity, Birkhoff's transitivity theorem, $J$-set, almost topological transitivity}

\begin{abstract}
In this note we prove a Birkhoff type transitivity theorem for continuous maps acting on non-separable completely metri\-zable spaces and we give some applications for
dynamics of bounded linear operators acting on complex Fr\'{e}chet spaces. Among them we show that any positive power and any unimodular multiple of a topologically
transitive linear operator is topologically transitive, generalizing similar results of S.I. Ansari and F. Le\'{o}n-Saavedra - V. M\"{u}ller for hypercyclic operators.
\end{abstract}

\maketitle

\section{Introduction and basic notions}
The scope of this paper is to provide a tool for studying topologically transitive operators acting on non-separable Fr\'{e}chet spaces by using technics and known
results from the theory of hypercyclic operators. This tool is the following theorem in which we show that each vector in the underlying space $X$ is contained in
``many" closed invariant subspaces such that the restriction of the operator to each one of them is hypercyclic:

\begin{Theorem}[Theorem \ref{hypsub}]
Let $T:X\to X$ be a topologically transitive operator acting on a completely metrizable vector space $X$ and let $y$ be a vector of $X$. There exists a dense
$G_{\delta}$ subset $D$ of $X$ such that for each $z\in D$ there exists a $T$-invariant (separable) closed subspace $Y_z$ of $X$ with $y,z\in Y_z$ such that the
restriction of $T$ to $Y_z$, $T:Y_z\to Y_z$, is hypercyclic.
\end{Theorem}

The previous theorem is derived from a Birkhoff type transitivity theorem for topologically transitive continuous selfmaps of a (not necessarily separable) completely
metrizable space. More precisely we show the following:

\begin{Theorem}[Theorem \ref{toptrans}]
Let $X$ be a completely metrizable space which has no isolated points. Let $T:X\to X$ be a continuous map acting on $X$  and let $x\in X$. If $T$ is topologically
transitive there exists a dense $G_{\delta}$ subset $D$ of $X$ with the following properties:
\begin{enumerate}
\item[(i)] Every point $z\in D$ is recurrent, that is there exists a strictly increasing sequence of positive integers $\{ k_n\}$ such that $T^{k_n}z\to z$.
\item[(ii)] The point $x$ belongs to the orbit closure $\overline{O(z,T)}$ of $z$ for every $z\in D$; in particular $x$ belongs to the limit set $L(z)$ of $z$
for every $z\in D$, i.e. for each $z\in D$ there exists a strictly increasing sequence of positive integers $\{ m_n\}$ such that $T^{m_n}z\to x$.
\end{enumerate}
\end{Theorem}

This result comes by ``localizing" Birkhoff's transitivity theorem; see Theorem \ref{localtr} and Corollary \ref{localtr2}.  This local behavior was hidden in the proof
of Birkhoff's transitivity theorem behind the use of a countable base of $X$. Recall that a continuous map $T:X\to X$ acting on a Hausdorff topological space $X$ is
called \textit{topologically transitive} if for every pair of non-empty open subsets $U,V$ of $X$ there exists a non-negative integer $n$ such that $T^nU\cap V\neq
\emptyset $. In case $X$ is completely  metrizable and separable, Birkhoff's transitivity theorem \cite{Bir} states that topological transitivity is equivalent to
\textit{hypercyclicity}, i.e. there exists a point $x$ in $X$ which has a dense orbit. The \textit{orbit} of $x\in X$ under $T$ is the set $O(x,T):=\{ T^n x\,;\,
n\in\mathbb{N}\cup \{ 0\}\}$. Hypercyclic operators on separable Fr\'{e}chet spaces have been a subject of extensive study in the last decades. Very good references are
the recent books \cite{BaMa} and \cite{GEPe}.

The applications we give in $\S3$ indicate the possibility to use the previous mentioned results to work on non-separable Fr\'{e}chet spaces by applying technics and
known results from the theory of hypercyclic operators. Namely, in $\S3$, we show the following theorem which generalizes similar results of S.I. Ansari \cite{Ansa} and
F. Le\'{o}n-Saavedra - V. M\"{u}ller \cite{LeoMu} for the case of a hypercyclic operator:

\begin{Theorem}[Theorem \ref{ALM}]
Let $T:X\to X$ be a topologically transitive operator acting on a complex Fr\'{e}chet space $X$ and let $x\in X$. Then
\begin{enumerate}
\item[(i)] The operator $T^p:X\to X$ is topologically transitive for every positive integer $p$ and there exists a dense subset $D$ of $X$ with the
following properties:
\begin{enumerate}
\item[(a)] Every vector $z\in D$ is a recurrent vector for the operator  $T^p$.
\item[(b)] The vector $x$ belongs to the limit set $L_{T^p}(z)$ for every positive integer $p$ and for every $z\in D$.
\end{enumerate}
\item[(ii)] If $\lambda$ is a complex number of modulus $1$ the operator $\lambda\, T:X\to X$ is topologically transitive and there exists a dense subset $D$ of $X$
with the following properties:
\begin{enumerate}
\item[(a)] Every vector $z\in D$ is a recurrent vector for the operator $\lambda\, T$.
\item[(b)] The vector $x$ belongs to the limit set $L_{\lambda\, T}(z)$ for every $|\lambda |=1$ and for every $z\in D$.
\end{enumerate}
\end{enumerate}
\end{Theorem}

We finish $\S3$ with a characterization, similar to that of H.N. Salas in \cite{Salas}, of topological transitivity of a backward unilateral weighted shift on $l_2(H)$,
where $H$ is a (not necessarily separable) Hilbert space, in terms of its weight sequence; see Proposition \ref{Salas}.

Finally, in $\S4$ we show that every continuous almost topologically transitive map acting on a completely metrizable space which has no isolated points is topologically
transitive; see Proposition \ref{altra}. Recall that a continuous map $T$ acting on a Hausdorff topological space $X$ is called \textit{almost topologically transitive}
if for every pair of non-empty open sets $U,V\subset X$ there exists a non-negative integer $n$ such that $T^nU\cap V\neq \emptyset$ or $T^nV\cap U\neq \emptyset$.

\section{A Birkhoff type transitivity theorem for non-separable completely metrizable spaces}
Birkhoff's transitivity theorem can be reformulated using some very useful concepts from the theories of Dynamical Systems and Topological Transformation Groups, namely
the concept of the \textit{limit set} of a point $x\in X$:

\[
\begin{split}
L(x)=\{&y\in X:\,\mbox{ there exists a strictly increasing sequence}\\
       &\mbox{of positive integers}\,\,\{k_{n}\}\mbox{ such that }\,T^{k_{n} }x\rightarrow y\}
\end{split}
\]
that describes the limit behavior of the orbit $O(x,T)$ and the generalized (prolongational) limit set of $x\in X$ that describes the asymptotic behavior of the
orbits of nearby points to $x\in X$:

\[
\begin{split}
J(x)=\{ &y\in X:\,\mbox{ there exist a strictly increasing sequence of positive}\\
&\mbox{integers}\,\{k_{n}\}\,\mbox{and a sequence }\,\{x_{n}\}\subset X\,\mbox{such that}\, x_{n}\rightarrow x\,
\mbox{and}\\
&T^{k_{n}}x_{n}\rightarrow y\}.
\end{split}
\]

The limit and the extended limit sets have their roots in the Stability Theory of Dynamical Systems in which they are mainly used to describe the Lyapunov and the
asymptotic stability of an equilibrium point.  They are $T$-invariant and closed subsets of $X$, see e.g. \cite[Proposition 2.6]{cosma1}. In view of the concept of limit
sets, a bounded linear operator $T:X\to X$ acting on a separable Fr\'{e}chet space $X$ is hypercyclic if and only if $L(x)=X$ for some non-zero vector $x\in X$. And
Birkhoff's transitivity theorem says that $T$ is hypercyclic if and only if $J(x)=X$ for every $x\in X$.

\begin{remark}
If $X$ is a completely metrizable space which has no isolated points (hence $X$ is uncountable), the orbit of a point $x\in X$ (which is a countable set) is dense in $X$
if and only if $L(x)=X$. In a similar way, topological transitivity, i.e.  $J(x)\cup O(x,T)=X$ for every $x\in X$, means that $J(x)=X$ for every $x\in X$ since $J(x)$ is
a closed subset of $X$. In general, if $X$ is a completely metrizable space, one can use instead of the limit set $L(x)$  the orbit closure $\overline{O(x,T)}=O(x,T)\cup
L(x)$ and instead of the generalized limit set $J(x)$ the union $O(x,T)\cup J(x)=:D(x)$ (the concept of the $D$-set comes also from the Stability Theory of Dynamical
Systems and it is called the prolongation of the orbit $O(x,T)$, see e.g. \cite{BaSz}).
\end{remark}

The following theorem, which can be seen as a ``localized" Birkhoff's transitivity theorem, plays a key role in our approach to topological transitivity. The symbol
$B(x,\varepsilon )$ stands for the open ball centered at $x$ with radius $\varepsilon >0$. Recall that a point $x\in X$ is called \textit{recurrent} if $x\in L(x)$.

\begin{theorem} \label{localtr}
Let $X$ be a completely metrizable space and let $T:X\to X$ be a continuous map acting on $X$. Assume that there are points $x,y\in X$ and a positive number $\varepsilon
>0$ such that $x\in J(w)$ for every $w\in B(y,\varepsilon )$. Then
\begin{enumerate}
\item[(i)] there is a point $z\in B(y,\varepsilon )$ such that $x\in L(z)$;

\item[(ii)] if, moreover, every point $w\in B(y,\varepsilon )$ is non-wandering, i.e. $w\in J(w)$, there is a recurrent point $z\in B(y,\varepsilon )$ such that $x\in L(z)$.
\end{enumerate}
\end{theorem}
\begin{proof}
We give only the proof of claim (ii) since the proof of claim (i) is similar. Since $x\in J(y)$ there exist a point $y_1\in B(y,\varepsilon )$, a positive integer $k_1$
and an open ball $B(y_1,\varepsilon_1 )\subset B(y,\varepsilon )$, with $\varepsilon_1 <1$ such that $T^{k_1}B(y_1,\varepsilon_1 )\subset B(x,1)$. And since $y_1\in
J(y_1)$ there exist a point $w_1\in  B(y_1,\varepsilon_1 )$, an open ball $B(w_1,r_1) \subset B(y_1,\varepsilon_1 )$, for some $r_1>0$, and a positive integer $m_1$ such
that $T^{m_1} B(w_1,r_1)\subset B(y_1,\varepsilon_1 )$. Note that $x\in J(w_1)$ since $x\in J(w)$ for every $w\in B(y,\varepsilon )$. Thus, there exist a point $y_2\in
B(w_1,r_1 )$, a positive integer $k_2>k_1$ and an open ball $B(y_2,\varepsilon_2 )\subset B(w_1,r_1 )$ with $\varepsilon_2 <1/2$ such that $T^{k_2}B(y_2,\varepsilon_2
)\subset B(x,1/2)$. Now we can proceed as before. Since $y_2\in J(y_2)$ there exist a point $w_2\in  B(y_2,\varepsilon_2 )$, an open ball $B(w_2,r_2) \subset
B(y_2,\varepsilon_2 )$, for some $r_2>0$, and a positive integer $m_2>m_1$ such that $T^{m_2} B(w_2,r_2)\subset B(y_2,\varepsilon_2 )$. Proceeding by induction we can
find two sequences of points $\{ y_n\}$ and $\{ w_n\}$, two strictly increasing sequences of positive integers $\{ k_n\}$ and $\{ m_n\}$ and two sequences of positive
numbers $r_n \leq \varepsilon_n <\frac{1}{2^n}$ with the following properties:
\begin{enumerate}
\item[(a)] $B(y_{n+1},\varepsilon_{n+1} )\subset B(w_n,r_n )\subset B(y_n,\varepsilon_n)$,

\item[(b)] $T^{k_{n+1}}B(y_{n},\varepsilon_{n} )\subset B(x,1/(n+1))$ and

\item[(c)] $T^{m_n} B(w_{n},r_n)\subset B(y_n,\varepsilon_n )$,
\end{enumerate}
for every $n\in\mathbb{N}$. Since $X$ is a complete metric space and $r_n \leq \varepsilon_n <\frac{1}{2^n}$ then
\[
\bigcap_n B(y_n,\varepsilon_n )\, =\, \bigcap_n B(w_n,r_n )\, = \, \{z \}, \quad \mbox{for some}\,\, z\in X.
\]
Therefore, $T^{k_n}z\to x$ by property (b) and $T^{m_n}z\to z$ by properties (a) and (c).
\end{proof}

\begin{corollary} \label{localtr2}
Let $(X,d)$ be a complete metric space with metric $d$ and let $T:X\to X$ be a continuous map acting on $X$. Assume that there are points $x,y\in X$ and a positive
number $\varepsilon
>0$ such that $x\in J(w)$ for every $w\in B(y,\varepsilon )$. Then
\begin{enumerate}
\item[(i)] There is a dense $G_{\delta}$ subset $D$ of $\overline{B(y,\varepsilon )}$ such that $x\in L(z)$ for every $z\in D$.

\item[(ii)] If, moreover, every point in $B(y,\varepsilon )$ is non-wandering, i.e. $w\in J(w)$,  there is a dense $G_{\delta}$ subset $D$ of $\overline{B(y,\varepsilon )}$
such that $x\in L(z)$  and $z\in L(z)$ for every $z\in D$.
\end{enumerate}
\end{corollary}
\begin{proof}
By a diagonal procedure, see \cite[Lemma 2.4]{cosma2}, $x\in J(w)$ for every $w\in \overline{B(y,\varepsilon )}$ since $x\in J(w)$ for every $w\in B(y,\varepsilon )$.
Moreover, every point of the closure of $B(y,\varepsilon )$ is non-wandering. We do that in order to apply the Baire category theorem and Theorem \ref{localtr} for the
complete metric space $\overline{B(y,\varepsilon )}$. Now, in (i), let $D:=\{ z\in\overline{B(y,\varepsilon )}\,\,\mbox{such that}\,\,x\in L(z)\}$. Then $D$ can be
written as a countable intersection of open and dense subset of $\overline{B(y,\varepsilon )}$ in the following way:
\[
D=\bigcap_{s,n\in\mathbb{N}} \bigcup_{m>n}T^{-m}B(x,\frac{1}{s})\cap \overline{B(y,\varepsilon )}.
\]
Similarly, in (ii) let $D:=\{ z\in\overline{B(y,\varepsilon )}\,\,\mbox{such that}\,\,x\in L(z)\,\,\mbox{and}\,\, z\in L(z)\}$. Then
\[
D=\bigcap_{s,n\in\mathbb{N}} \left ( \bigcup_{m>n} T^{-m}B(x,\frac{1}{s}) \cap \bigcup_{l>n}\{ w\in X:\,  d(T^l w,w) < \frac{1}{s}\} \right ) \cap
\overline{B(y,\varepsilon)}.
\]
Hence, in both cases (i) and (ii), the set $D$ is a dense $G_{\delta}$ subset of $\overline{B(y,\varepsilon )}$.
\end{proof}

The following two results are immediate corollaries of Theorem \ref{localtr}. They can also be found in \cite[Satz 1.2.2]{Gro1} and in a more general form in
\cite[Theorem 26]{BENP}. We are grateful to K.-G. Grosse-Erdmann for pointing out this to us.

\begin{theorem}[A Birkhoff type transitivity theorem for non-separable completely metrizable spaces] \label{toptrans}
Let $X$ be a completely metrizable space which has no isolated points. Let $T:X\to X$ be a continuous map acting on $X$ and let $x\in X$. If $T$ is topologically
transitive there exists a dense $G_{\delta}$ subset $D$ of $X$ with the following properties:
\begin{enumerate}
\item[(i)] Every point $z\in D$ is recurrent.
\item[(ii)] The point $x$ belongs to the limit set $L(z)$ for every $z\in D$.
\end{enumerate}
\end{theorem}

\begin{corollary} \label{countablepoints}
Let $X$ be a completely metrizable space which has no isolated points. Let $T:X\to X$ be a continuous map acting on  $X$ and let $A$ be a countable subset of $X$. If $T$
is topologically transitive there exists a dense $G_{\delta}$ subset $D$ of $X$ with the following properties:
\begin{enumerate}
\item[(i)] Every point $z\in D$ is recurrent.
\item[(ii)] The set $\overline{A}$ is a subset of $L(z)$ for every $z\in D$.
\end{enumerate}
\end{corollary}

Note that if $X$ is completely metrizable and separable then we can recover the original transitivity theorem of Birkhoff by applying the previous corollary.

\section{Applications to Linear Dynamics}
In this section we will use Theorem \ref{hypsub} and known results from the theory of hypercyclic operators to derive similar results for the non-separable case.

\medskip

We quote from \cite[Preface]{GEPe}, ``...Some of the deepest, most beautiful and most useful results from linear dynamics is Ansari's theorem on the powers of
hypercyclic operators and the Le\'{o}n-M\"{u}ller theorem on the hypercyclicity of unimodular multiples of hypercyclic operators...". More precisely, Ansari showed in
\cite{Ansa} that if $T$ is a hypercyclic operator then any positive power of $T$ is also a hypercyclic operator with the same set of hypercyclic vectors. And
Le\'{o}n-Saavedra and M\"{u}ller  showed in \cite{LeoMu} that if $T$ is a hypercyclic operator and $\lambda$ is a complex number of modulus $1$ then $\lambda\, T$ is
also a hypercyclic operator with the same hypercyclic vectors. The reason that we can obtain similar results in the more general setting of a topologically transitive
operator is because each vector in $X$ is contained in ``many"  closed invariant subspaces  such that the restriction of the operator to each one of them is hypercyclic:

\begin{theorem} \label{hypsub}
Let $T:X\to X$ be a topologically transitive operator acting on a completely metrizable vector space $X$ and let $y$ be a vector of $X$. There exists a dense
$G_{\delta}$ subset $D$ of $X$ such that for each $z\in D$ there exists a $T$-invariant (separable) closed subspace $Y_z$ of $X$ with $y,z\in Y_z$ such that the
restriction of $T$ to $Y_z$, $T:Y_z\to Y_z$,  is hypercyclic.
\end{theorem}
\begin{proof}
Let $W_0$ denote the closed linear span of the orbit $O(y,T)$, that is $W_0$ is the closure of the subspace $\{ P(T)y\,:\, $P$\,\,\mbox{polynomial}\}$. Since $T$ is
topologically transitive then, by Corollary \ref{countablepoints}, there exists a dense $G_{\delta}$ subset $D$ of $X$ which consists of recurrent vectors such that
$W_0\subset \overline{O(x,T)}$ for each $x\in D$. Now, choose a vector $z\in D$ and set $x_1:=z$. Thus, $W_0\subset \overline{O(x_1,T)}$ and $x_1\in L(x_1)$. Let $W_1$
be the closed linear span of the orbit $O(x_1,T)$. By the same corollary there exists a vector $x_2\in X$ such that $W_0\subset \overline{O(x_1,T)}\subset W_1\subset
\overline{O(x_2,T)}$ and $x_2\in L(x_2)$. Proceeding by induction, there is a sequence of vectors $\{ x_n\}_{n\in\mathbb{N}}$ in $X$ such that $W_{n-1}\subset
\overline{O(x_n,T)}\subset W_n$, where  $W_n$ denotes the closed linear span of the orbit $O(x_n,T)$, and $x_n\in L(x_n)$ for every $n\in\mathbb{N}$. Set
$Y_z:=\overline{\bigcup_{n=0}^{+\infty} W_n}$. Obviously, $Y_z$ is a $T$-invariant, closed and separable subspace of $X$. Now, let us show that the restriction of $T$ to
$Y_z$, $T:Y_z\to Y_z$, is hypercyclic. Let $U,V\subset Y_z$ be two non-empty relatively open subsets of $Y_z$. Since $Y_z:=\overline{\bigcup_{n=1}^{+\infty}
\overline{O(x_n,T)}}$ and $ \overline{O(x_n,T)}\subset  \overline{O(x_{n+1},T)}$ for every $n\in\mathbb{N}$, there exist a positive integer $p$ and some non-negative
integers $k$ and $m$ such that $T^kx_p\in U$ and $T^mx_p\in V$ (note that the vector $x_p$ is common for both $U$ and $V$). Since $x_p$ is a recurrent vector and $V$ is
open in $T_z$ we may assume that $m>k$. Now note that $T^{m-k}(T^kx_p)=T^mx_p\in V$. Thus, $T^{m-k}U\cap V\neq\emptyset$ and hence $T:Y_z\to Y_z$ is topologically
transitive. Therefore, by Birkhoff's transitivity theorem, $T:Y_z\to Y_z$ is hypercyclic and since $z$ was an arbitrary vector of $D$, which is a dense $G_{\delta}$
subset of $X$, the proof is completed.
\end{proof}

Now we are ready to show Ansari's and Le\'{o}n-M\"{u}ller theorems for topologically transitive operators:

\begin{theorem} \label{ALM}
Let $T:X\to X$ be a topologically transitive operator acting on a complex Fr\'{e}chet space $X$ and let $x\in X$. Then
\begin{enumerate}
\item[(i)] The operator $T^p:X\to X$ is topologically transitive for every positive integer $p$ and there exists a dense subset $D$ of $X$ with the
following properties:
\begin{enumerate}
\item[(a)] Every vector $z\in D$ is a recurrent vector for the operator  $T^p$.
\item[(b)] The vector $x$ belongs to the limit set $L_{T^p}(z)$ for every positive integer $p$ and for every $z\in D$
\end{enumerate}
\item[(ii)] If $\lambda$ is a complex number of modulus $1$ the operator $\lambda\, T:X\to X$ is topologically transitive and there exists a dense subset $D$ of $X$
with the following properties:
\begin{enumerate}
\item[(a)] Every vector $z\in D$ is a recurrent vector for the operator $\lambda\, T$.
\item[(b)] The vector $x$ belongs to the limit set $L_{\lambda\, T}(z)$ for every $|\lambda |=1$ and for every $z\in D$.
\end{enumerate}
\end{enumerate}
\end{theorem}
\begin{proof}
Let $U,V$ be two non-empty open subsets of $X$ and let $x\in U$. By Theorem \ref{hypsub} there is a vector $z\in V$ and a $T$-invariant separable closed subspace $Y_z$
of $X$ with $x,z\in Y_z$ such that the restriction of $T$ to $Y_z$, $T:Y_z\to Y_z$, is hypercyclic. Hence, we can apply the original theorems of Ansari and
Le\'{o}n-M\"{u}ller for the restriction $T:Y_z\to Y_z$. Thus, the operators $T^p:Y_z\to Y_z$ and $\lambda\, T:Y_z\to Y_z$ are hypercyclic and share the same set of
hypercyclic vectors with $T:Y_z\to Y_z$. So we can find a (common) hypercyclic vector $w\in Y_z$ for the operators $T^p:Y_z\to Y_z$ and $\lambda\, T:Y_z\to Y_z$ such
that the vectors $x,z$ belong to the limit sets $L_{T^p}(w)$ and $L_{\lambda\, T}(w)$ respectively. Therefore, like in the proof of Theorem \ref{hypsub}, the operators
$T^p:X\to X$ and $\lambda\, T:X\to X$ are topologically transitive. In both cases (i) and (ii), let $D$ be the union of the sets of hypercyclic vectors of the family of
operators $T:Y_z\to Y_z$ for $z$ in a dense $G_{\delta}$ subset of $X$. Thus, $D$ is dense in $X$.
\end{proof}

\begin{remark}
Linearity of the map under consideration is essential in the proof of Theorem \ref{ALM}. It is evident, that even for real piecewise linear maps, a power of a
topologically transitive map may not be topologically transitive; see e.g. \cite{Banks}.
\end{remark}

We finish this section with a characterization, similar to that of H.N. Salas in \cite{Salas}, of topological transitivity of a backward unilateral weighted shift on
$l_2(H)$, where $H$ is a (not necessarily separable) Hilbert space, in terms of its weight sequence. The motivation of this problem comes from a work of T. Berm\'{u}dez
and N.J. Kalton \cite{BeKa}. In this paper they showed that spaces like $l^{\infty}(\mathbb{N})$ and $l^{\infty}(\mathbb{Z})$ do not support topologically transitive
operators. On the other hand, following an example suggested by J.H. Shapiro, they showed that every non-separable Hilbert space supports a topologically transitive
operator. To do that they wrote $H$ as $l_2(X)$ for some  Hilbert space $X$ of the same density character with $H$, i.e. a space with the same least cardinality of a
dense subset. Then they showed that the operator $2B$, where $B$ is the backward unilateral shift on $l_2(X)$, is topologically transitive. Hence, a problem that
naturally arises is to give a characterization of topological transitivity of a backward unilateral weighted shift on $l_2(H)$, where $H$ is a (not necessarily
separable) Hilbert space, in terms of its weight sequence. This characterization is similar to that of H.N. Salas in \cite{Salas}:

\begin{proposition} \label{Salas}
Let $H$ be Hilbert space and let $T:l_2(H)\rightarrow l_2(H)$ be a unilateral backward weighted shift with positive weight sequence $(w_n)$. The following are
equivalent:
\begin{enumerate}
\item[(i)] $T$ is topologically transitive.

\item[(ii)] There exists a non-trivial $T$-invariant (separable) closed subspace $Y\subset l_2(H)$ on which the restriction of $T$ to $Y$, $T:Y\to Y$, is hypercyclic.

\item[(iii)] The restriction $T:Y\to Y$ to any $T$-invariant (separable) closed subspace $Y\subset l_2(H)$ is hypercyclic.

\item[(iv)] $\limsup_{n\to\infty} (w_1\cdots w_n)\to +\infty$.
\end{enumerate}
\end{proposition}
\begin{proof}
Property (i) implies (ii) by Theorem \ref{hypsub}. From Salas characterization of hypercyclic unilateral backward shifts in \cite{Salas}, claim (ii) implies (iii). From
the same theorem property (iv) implies (iii) and vice versa. Finally, it is very easy to see that claim (iii) implies (i).
\end{proof}

\section{Almost topologically transitive operators are topologically transitive}
Among the concepts that generalize, or are closely related to, topological transitivity is the concept of a topological semi-transitive operator; see e.g. \cite{RoTro}
for the case of operator algebras. Recall that an operator $T:X\to X$ acting on a topological vector space $X$ is called \textit{topologically semi-transitive} if for
every two non-zero vectors $x,y\in X$ we have that $x\in \overline{O(y,T)}$ or $y\in \overline{O(x,T)}$. There are some disadvantages in this concept for the case of the
action of a semigroup of linear operators. One disadvantage is that  topological semi-transitivity is not a topological concept. Another disadvantage is that topological
semi-transitivity is very restrictive, since even a chaotic operator is not topologically semi-transitive (to see that just consider two non-zero periodic vectors). A
concept that looks more attractive to generalize topological transitivity is the concept of almost topological transitivity:

\begin{definition} \label{defalmost}
Let $X$ be a Hausdorff topological space and let $T:X\to X$ be a continuous map acting on $X$. We call $T$ \textit{almost topologically transitive} if for every pair of
non-empty open sets $U,V\subset X$ there exists a non-negative integer $n$ such that $T^nU\cap V\neq \emptyset$ or $T^nV\cap U\neq \emptyset$. That is $x\in
D(y):=O(y,T)\cup J(y)$ or $y\in D(x):=O(x,T)\cup J(x)$ for every $x,y\in X$.
\end{definition}

Note that if $X$ is a completely metrizable space which has no isolated points the orbit $O(x,T)$ has empty interior for every $x\in X$. In this case the following
proposition shows that a continuous map $T:X\to X$ is almost topologically transitive if $x\in J(y)$ or $y\in J(x)$ for every $x,y\in X$.

\begin{proposition} \label{altraJ}
Let $X$ be a completely metrizable space which has no isolated points. A continuous map $T:X\to X$ is almost topologically transitive if $x\in J(y)$ or $y\in J(x)$ for
every $x,y\in X$.
\end{proposition}
\begin{proof}
We proceed by contradiction. Assume that there exists a pair of (not necessarily distinct) vectors $x,y\in X$ such that $x\notin J(y)$ and $y\notin J(x)$. Hence, $x\in
O(y,T)$ or $y\in O(x,T)$ since $T$ is almost topologically transitive. Without loss of generality we may assume that $x\in O(y,T)$. Since the orbit $O(y,T)$ has empty
interior there is a sequence $\{x_n\}_{n\in\mathbb{N}}$ in $X$ such that $x_n\to x$ and $x_n\notin O(y,T)$ for every $n\in\mathbb{N}$. Thus, the following hold: (a)
$x_n\in J(y)$ for infinitely many $n\in\mathbb{N}$, so $x\in J(y)$ since $J(y)$ is a closed subset of $X$ which is a contradiction or (b) $y\in D(x_n):=O(x_n,T)\cup
J(x_n)$ eventually. In that case, using a diagonal procedure similar to the one in \cite[Lemma 2.4]{cosma2}, it follows that $y\in D(x):=O(x,T)\cup J(x)$. This implies
that $y\in O(x,T)$ since $y\notin J(x)$. But also $x\in O(y,T)$, hence $O(x,T)=O(y,T)$. So, if $x,y$ are distinct vectors then $O(x,T)=O(y,T)$ is a periodic orbit, thus
$x\in J(y)$ (and $y\in J(x)$) which is a contradiction. If $x=y$, and since $x_n\notin O(y,T)=O(x,T)$, we may assume that $x_n\neq x$ for every $n\in\mathbb{N}$. Hence,
as we showed above for the case of two distinct points, $x_n\in J(x)$ for infinitely many $n\in\mathbb{N}$ or $x\in J(x_n)$ eventually. In both cases, $x\in J(x)=J(y)$,
by \cite[Lemma 2.4]{cosma2}, which is again a contradiction.
\end{proof}

The above proposition and Corollary \ref{localtr2} imply the following:

\begin{proposition} \label{altra}
Let $X$ be a completely metrizable space which has no isolated points and let $T:X\to X$ be an almost topologically transitive selfmap of $X$. Then $T$ is topologically
transitive.
\end{proposition}
\begin{proof}
Note that $x\in J(x)$ for every $x\in X$ since $T$ is almost topologically transitive. Now, let us show that $J(x)=X$ for every $x\in X$. We argue by contradiction.
Assume that $J(x)\neq X$ for some $x\in X$. Hence, $x\in J(y)$ for every $y$ is in the complement of $J(x)$. Therefore, by Corollary \ref{localtr2}, there exist a point
$y\notin J(x)$ and two strictly increasing sequences of positive integers $\{ k_n\}$ and $\{ m_n\}$ such that $T^{k_n}y\to x$ and $T^{m_n}y\to y$. We may assume that
$m_n-k_n\to +\infty$. Thus, $T^{m_n-k_n}(T^{k_n}y)=T^{m_n}y\to x$ and $T^{k_n}y\to x$, hence $y\in J(x)$ which is a contradiction.
\end{proof}

\noindent \textbf{Acknowledgements.} We would like to thank Fr\'{e}d\'{e}ric Bayart, George Costakis, Karl-Goswin Grosse-Erdmann and Alfredo Peris for their very useful
comments.

\end{document}